\documentclass[a4pape, reqno,12pt,oneside]{amsart}

\usepackage{latexsym}
\usepackage{amssymb}
\usepackage{graphicx}
\usepackage{wrapfig}
\usepackage{amsthm}
\usepackage{float}
\usepackage{epsfig}
\usepackage{multirow}
\usepackage[toc,page]{appendix}
\usepackage[footnotesize]{caption}
\RequirePackage{color}
\usepackage{cite}

\numberwithin{figure}{section}
\usepackage{subfigure}

\usepackage{amscd,enumerate}
\usepackage{amsfonts}
\usepackage{enumitem} 

\input xy
\xyoption{all}

\usepackage{multicol}
\usepackage{hyperref}
\hypersetup{ colorlinks,
linkcolor=blue,
filecolor=green,
urlcolor=blue,
citecolor=blue }

\newtheorem{lemma}{Lemma}[section]
\newtheorem{corollary}[lemma]{Corollary}
\newtheorem{theorem}[lemma]{Theorem}
\newtheorem{prop}[lemma]{Proposition}
\newtheorem{remark}[lemma]{Remark}

\newtheorem{exa}[lemma]{Example}

\newcommand{\uloopr}[1]{\ar@'{@+{[0,0]+(-4,5)}@+{[0,0]+(0,10)}@+{[0,0] +(4,5)}}^{#1}}
\newcommand{\uloopd}[1]{\ar@'{@+{[0,0]+(5,4)}@+{[0,0]+(10,0)}@+{[0,0]+
(5,-4)}}^{#1}}
\newcommand{\dloopr}[1]{\ar@'{@+{[0,0]+(-4,-5)}@+{[0,0]+(0,-10)}@+{[0,
0]+(4,-5)}}_{#1}}
\newcommand{\dloopd}[1]{\ar@'{@+{[0,0]+(-5,4)}@+{[0,0]+(-10,0)}@+{[0,0
]+(-5,-4)}}_{#1}}
\newcommand{\luloop}[1]{\ar@'{@+{[0,0]+(-8,2)}@+{[0,0]+(-10,10)}@+{[0,
0]+(2,2)}}^{#1}}

\usepackage{array}
\newcolumntype{L}[1]{>{\raggedright\let\newline\\\arraybackslash\hspace{0pt}}m{#1}}
\newcolumntype{C}[1]{>{\centering\let\newline\\\arraybackslash\hspace{0pt}}m{#1}}
\newcolumntype{R}[1]{>{\raggedleft\let\newline\\\arraybackslash\hspace{0pt}}m{#1}}

\def\K{\mathbb{K}}
\def\T {\mathcal{T}}
\def\A{\mathcal{A}}
\newcommand{\Der}{{\rm Der}}
\newcommand{\chart}{{\rm char}}

\usepackage{tikz}
\usepackage{tkz-berge}
\usetikzlibrary{arrows,snakes,backgrounds,trees}
\usetikzlibrary{matrix,decorations.pathreplacing,positioning}
\usepackage{stmaryrd}
\usepackage{tkz-graph}
\usetikzlibrary{decorations.markings}
\tikzstyle directed=[postaction={decorate,decoration={markings,
    mark=at position .65 with {\arrow{stealth}}}}]
\tikzstyle reverse directed=[postaction={decorate,decoration={markings,
    mark=at position .65 with {\arrowreversed{stealth};}}}]

\begin{document}


\author[Yolanda Cabrera Casado]{Yolanda Cabrera Casado}
\address{Yolanda Cabrera Casado:  Departamento de \'Algebra Geometr\'{\i}a y Topolog\'{\i}a, Fa\-cultad de Ciencias, Universidad de M\'alaga, Campus de Teatinos s/n. 29071 M\'alaga, Spain.}
\email{yolandacc@uma.es}

\author[Paula Cadavid]{Paula Cadavid}
\address{Paula Cadavid: Universidade Federal do ABC, Avenida dos Estados, 5001-Bangu-Santo Andr\'e - SP, Brazil.}
\email{pacadavid@gmail.com}

\author[Mary Luz Rodi\~no Montoya]{Mary Luz Rodi\~no Montoya}
\address{Mary Luz Rodi\~no Montoya: Instituto de Matem\'aticas,  Universidad de Antioquia, Calle 67 N$^{\circ}$ 53-108, Medell\'in, Colombia.}
\email{mary.rodino@udea.edu.co }

\author[Pablo M. Rodriguez ]{Pablo M. Rodriguez}
\address{Pablo M. Rodriguez: Centro de Ci\^encias Exatas e da Natureza, Universidade Federal de Pernambuco, Av. Prof. Moraes Rego, 1235 - Cidade Universit\'aria - Recife - PE, Brazil.}
\email{pablo@de.ufpe.br}

\thanks{
Part of this work was carried out during a visit of Y.C., M.L.R and P.M.R. at the Universidade Federal do ABC - UFABC in Santo Andr\'e, SP, Brazil. They are grateful for their hospitality and support. The authors also thank the Funda\c{c}\~ao Getulio Vargas - FGV, Rio de Janeiro, RJ, Brazil, host organization of the ILAS 2019 conference, for their hospitality. This work was supported by Funda\c{c}\~ao de Amparo \`a Pesquisa do Estado de S\~ao Paulo - FAPESP (Grants 2016/11648-0, 2017/10555-0, 2018/06925-0), and Conselho Nacional de Desenvolvimento Cient\'ifico e Tecnol\'ogico - CNPq (Grant 304676/2016-0).
\newline
}

\title[]{On the characterization of the space of derivations\\ in evolution algebras}

\begin{abstract}
We study the space of derivations for some finite-dimensional evolution algebras, depending on the twin partition of an associated directed graph. For evolution algebras with a twin-free associated graph we prove that the space of derivations is zero. For the remaining families of evolution algebras we obtain sufficient conditions under which the study of such a space can be simplified. We accomplish this task by identifying the null entries of the respective derivation matrix. Our results suggest how strongly the associated graph's structure impacts in the characterization of derivations for a given evolution algebra. Therefore our approach constitutes an alternative to the recent developments in the research of this subject. As an illustration of the applicability of our results we provide some examples and we exhibit the classification of the derivations for non-degenerate irreducible $3$-dimensional evolution algebras.
\end{abstract}
\maketitle


\section{Introduction}

The evolution algebras became an interesting class of non-associative genetic algebras because of their many applications and connections to other areas of mathematics. The first reference of a Theory of Evolution Algebras is due to Tian and Vojtechovsky, see \cite{tv}, who in 2006 state the first properties for these mathematical structures. Further on, Tian, in his seminal work \cite{tian}, gave an interesting correspondence between this and other subjects of mathematics like, for example, the Theory of Discrete-time Markov Chains. We refer the reader to \cite{YMV,YMV2,PMP,PMP2,PMP3,camacho/gomez/omirov/turdibaev/2013,COT,cardoso,casas/ladra/omirov/rozitov/2013,casas/ladra/rozitov/2011,Elduque/Labra/2015,Elduque/Labra/2019,falcon,paniello,tian2} and references therein for an overview of recent results, applications, and interesting open problems. 

\smallskip
In order to state the first definitions let $\Lambda:=\{1,\ldots,n\}$. An $n$-dimensional $\K$-algebra is called an \emph{evolution algebra} if it  admits a basis $B=\{e_i \}_{i \in \Lambda}$ such that $e_{i} \cdot e_{j}=0$ whenever $i \neq j$. A basis with this property is known as \emph{natural basis}. The scalars $\omega_{ij} \in \K$ such that \[e_i^2=e_{i}\cdot e_{i} =\displaystyle\sum_{k \in \Lambda} \omega_{ik} e_k\] are called the \emph{structure constants} of $\A$ relative to $B$ and the matrix $M_{B}=(\omega_{ik})$ is called the \emph{structure matrix} of $\A$ relative to $B$. Our purpose is to study the derivations of an evolution algebra. Given an (evolution) algebra $\A$, a  \emph{derivation} of $\A$ is a linear map $d: \A \rightarrow  \A$ such that 
\[d(u \cdot v)= d(u) \cdot  v + u \cdot d(v),\]
for all $u,v \in \A$. The space of all derivations of the (evolution) algebra $\A$ is denoted by $\Der(\A)$. In \cite{tian} it is observed that a linear map $d$ such that $d(e_i)=\sum_{k=1}^{n} d_{ik}e_k$ is a derivation of the evolution algebra $\A$ if, and only if, it satisfies the following conditions: 

\begin{eqnarray}
\omega_{jk}d_{ij}+ \omega_{ik}d_{ji}=0, &\,\,\, \text{ for }i,j,k\in\Lambda\text{ such that }i\neq j,\label{eq:der1}\\
\sum_{k=1}^{n} \omega_{ik}d_{kj}=2\omega_{ij}d_{ii}, &\,\,\, \text { for }i,j\in\Lambda,\label{eq:der2}
\end{eqnarray}

\noindent
where $M_B=(\omega_{ij})$ is the structure matrix of $\A$ relative to a natural basis $B$. Therefore,  \eqref{eq:der1} and \eqref{eq:der2} are the starting point whether one want to obtain a description of the space of derivations of an evolution algebra $\A$. We point out that such a space is a Lie algebra which may be used as a tool for studying the structure of the original algebra. For some genetic algebras, the space of all the derivations has already been described in \cite{costa1,costa2,gonshor,gonzales,holgate,Mukhamedov/Qaralleh/2014,peresi}. In the particular case of an evolution algebra, a complete characterization of such space is still an open question. We recommend to the reader the references \cite{Alsarayreh/Qaralleh/Ahmad/2017,PMP3, COT,cardoso, Elduque/Labra/2019,Mukhamedov/Khakimov/Omirov/Qaralleh/2019} for a good idea of the existing results about this topic.

\smallskip
In \cite{COT} the authors prove that the space of derivations of $n$-dimensional complex evolution algebras with non-singular matrices is zero, and they describe the space of derivations of evolution algebras with matrices of rank $n-1$. In \cite{Elduque/Labra/2019} the authors extend the description of the derivations of evolution algebras with non-singular matrices to the case of fields with any characteristic. Although the approaches considered in \cite{Alsarayreh/Qaralleh/Ahmad/2017,PMP3,cardoso} are different, they provide a useful contribution to the field. While \cite{cardoso} gives a complete characterization for the space of derivation on two-dimensional evolution algebras, and \cite{Alsarayreh/Qaralleh/Ahmad/2017,Mukhamedov/Khakimov/Omirov/Qaralleh/2019} consider the case of three-dimensional solvable and finito-dimensional nilpotent evolution algebras,  \cite{PMP3} do it for the case of evolution algebras associated to graphs. We point out that the later includes examples of algebras with singular matrices with any rank so they arguments can be seen as an alternative to the ones developed by \cite{COT,Elduque/Labra/2019}. Indeed the approach stated in \cite{PMP3} rely only on the structural properties of the considered graph. In this work we extend this approach to study the derivations of a given evolution algebra. We accomplish this task by exploring the fact that, as pointed by \cite{Elduque/Labra/2015}, any evolution algebra induces a directed graph. Moreover, we illustrate the applicability of our approach by characterizing the derivations of non-degenerate irreducible $3$-dimensional evolution algebras.

\smallskip
The paper is organized as follows. In Section 2 we review some of the standard notation of Evolution Algebras and Graph Theory, we state our main results, and we illustrate their applicability through examples. Section 3 is devoted to the statement of auxiliary results and the proofs of our main theorems.

\section{Main results}

\subsection{Preliminary definitions and notation}
 
In order to present our results we start with basic definitions and notation for directed graphs appearing in \cite{YMV}. A \emph{directed graph} is a $4$-tuple $E=(E^0, E^1, s_E, r_E)$ where $E^0$, $E^1$ are sets and $s_E, r_E: E^1 \to E^0$ are maps. The elements of $E^0$ are called the \emph{vertices} of $E$ and the elements of $E^1$ the \emph{arrows} or \emph{directed edges} of $E$. For
$f\in E^1$ the vertices $r(f)$ and $s(f)$ are called the \emph{range} and the \emph{source} of $f$, respectively. If $E^0$ and $E^1$ are both finite we say  that $E$ is \emph{finite}. A $v\in E^{0}$ is called \emph{sink} if it verifying that $s(f) \neq v$, for every $f\in E^1$.  
A \emph{path} or a \emph{path from $s(f_1)$ to $r(f_n)$} in  $E$, $\mu$, is a finite sequence of arrows $\mu=f_1\dots f_n$
such that $r(f_i)=s(f_{i+1})$ for $i\in\{1,\dots,(n-1)\}$. In this case we say that $n$ is the \emph{length} of the path $\mu$ and denote by $\mu^0$ the set of its vertices, i.e.,
$\mu^0:=\{s(f_1),r(f_1),\dots,r(f_n)\}$. Let $\mu = f_1 \dots f_n$ be a path in $E$. If  $ \vert\mu\vert =n\geq 1$, and if
$v=s(\mu)=r(\mu)$, then $\mu$ is called a \emph{closed path based at $v$}.
 If
$\mu = f_1 \dots f_n$ is a closed path based at $v$ and $s(f_i)\neq s(f_j)$ for
every $i\neq j$, then $\mu$ is called a \emph{cycle based at} $v$ or simply a \emph{cycle}. A cycle of length $1$ will be said to be a \emph{loop}. Given a finite graph $E$, its \emph{adjacency matrix} is the matrix $Ad_{E}=(a_{ij}) \in {\mathbb Z}^{(E^0\times E^0)}$
where $a_{ij}$ is the number of arrows from $i$ to $j$.

\smallskip
There are several ways to associate a graph to an evolution algebra (see \cite{YMV, Elduque/Labra/2015}). We consider the directed graph described in \cite{YMV} as follows. Given a natural basis $B=\{e_i \}_{i\in \Lambda}$ of an evolution algebra $\A$ and its structure matrix $M_B=(\omega_{ij})\in  {\rm M}_\Lambda(\mathbb K)$, consider the matrix
$P=(a_{ij})\in  {\rm M}_\Lambda(\mathbb K)$ such that $a_{ij}=0$ if $\omega_{ij}=0$ and $a_{ij}=1$ if $\omega_{ij}\neq 0$. The \emph{graph associated to the evolution algebra} $\A$ (relative to the basis $B$), denoted by $E_\A^{B}$ (or simply by $E$ if the algebra $\A$ and the basis $B$ are understood) is the directed graph whose adjacency matrix is given by $P=(a_{ij})$. 

\begin{exa}\label{EXA:EA1} \rm
Let $\A$ be an evolution algebra with natural basis $B=\{e_1,e_2,e_3\}$ such that 

$$
\begin{array}{ccl}
e_1^2 &=&2e_1 + e_2,\\[.2cm] 
e_2^2&=&-e_ 1 +3 e_3,\\[.2cm]
e_3^2&=&3 e_3. 
\end{array}
$$

The associated directed graph is given in Figure \ref{FIG:EA1}.

\begin{figure}[h!]
    \centering
    
    \begin{tikzpicture}

\draw [thick] (0,0) circle (7pt);
\draw (0,0) node[font=\footnotesize] {1};
\draw [thick] (2,0) circle (7pt);
\draw (2,0) node[font=\footnotesize] {2};
\draw [thick] (4,0) circle (7pt);
\draw (4,0) node[font=\footnotesize] {3};


\draw [thick, directed] (0,0.25) to [bend left=45] (2,0.25);
\draw [thick, directed] (0,0.25) to [out=140,in=220,looseness=8] (0,-0.25);
\draw [thick, directed] (2,-0.25) to [bend left=45] (0,-0.25);
\draw [thick, directed] (2,-0.25) to [bend left=315] (4,-0.25);
\draw [thick, directed] (4,0.25) to [out=40,in=320,looseness=8] (4,-0.25);

\end{tikzpicture}
    
    \caption{Associated directed graph of the evolution algebra of Example \ref{EXA:EA1}.}
    \label{FIG:EA1}
\end{figure}
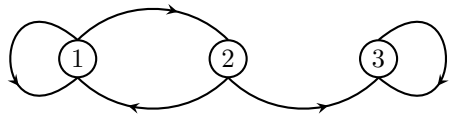

\end{exa}

The \emph{underlying graph} of the directed graph $G$ is the a graph which is obtained from $G$ after forgetting the orientation of their arrows. A directed graph $G$ is \emph{connected} if their underlying graph is connected. On the other hand, an evolution algebra $\A$ is \emph{reducible} if there exist two nonzero ideals $I$ and $J$ such that $\A=I \oplus J$. In other case, it is called \emph{irreducible}. An evolution algebra $\A$ is \emph{non-degenerate} if there is a natural basis $B$ such that $e_i^2 \neq 0$ for all $e_i \in B$. In general, the connectivity of the graph associated to an evolution algebra depends on the chosen natural basis, as \cite[Example 2.5]{Elduque/Labra/2015} shows.  It is well-known by \cite[Lemma 2.7]{Elduque/Labra/2015} the fact that a given algebra is non-degenerate does not depend on the chosen natural basis. Connectedness is related to irreducibly by \cite[Proposition 2.8]{Elduque/Labra/2015} which states that a finite non-degenerate evolution algebra $\A$  with a natural bases $B$ is irreducible if, and only if, the associated graph $E_\A^{B}$ is connected.

The following definitions can be found in \cite[Definition 3.1]{YMV}, but we include them here for the sake of completeness. Let $B=\{e_{i} \}_{i\in \Lambda}$ and $M_B=(\omega_{ij})$ be, respectively, a natural basis and the structure matrix  relative to $B$ of an evolution algebra $\A$. For any $i_{0}\in \Lambda$, the
\emph{first-generation descendants} of $i_{0}$ is the set given by $D^{1}(i_{0}):=\left\{k\in \Lambda \, \colon \, \omega _{i_{0}k}\neq 0\right\}.$ In addition, given a subset $U\subseteq \Lambda $, we let $D^{1}(U):=\{j\in \Lambda : j\in D^{1}(i) \text{ for some }i\in U\}$. If $E=(E^0,E^1,r,s)$ is the graph associated to $\A$, observe that $$D^1(i_0)= \{ k \in \Lambda: \exists f\in E^1 \, \vert \, s(f)=v_{i_0},\, r(f)=v_k\},$$

\noindent
with $v_{i_0},\, v_k \in E^0$. 

\smallskip
By analogy with Graph Theory we define the following notions. Let $\A$ be an evolution algebra with natural basis $B$ and let $i,j \in \Lambda$. We say that $i$ and $j$ are \emph{twins relative to $B$} if $D^1(i)=D^1(j)$. If $\A$ has no twins relative to $B$, that is, if
$$ \text{ for all } \, i,j \in \Lambda, \,i\neq j,\, D^1(i) \neq D^1(j), $$
\noindent
then we say that $\A$ is \emph{twin-free relative to} $B$ (see Example \ref{EXA:EA1}). We notice that by defining the relation $\sim_{t_B}$ on the set of indices $\Lambda$ by $i\sim_{t_B} j$ whether $i$ and $j$ are twins relative to $B$, then $\sim_{t_B}$ is an equivalence relation. An equivalence class of the twin relation $\sim_{t_B}$ is referred to as a {\it twin class relative to $B$}. In other words, the twin class of a index $i$ is the set $\{j\in \Lambda :i \sim_{t_B} j\}$. The set of all twin classes relative to $B$ of $\Lambda$ is denoted by $\Pi_B(\Lambda)$ and it is referred to as the \emph{twin partition relative to $B$} of $\Lambda$.
These definitions depend on the chosen natural basis as the following example shows.
\begin{exa}\label{EXA:EA2}
\rm
Let $\A$ be an evolution algebra with natural basis $B=\{e_1,e_2,e_3\}$ such that 
$$
\begin{array}{ccl}
e_1^2&=&e_2,\\[.2cm]
e_2^2&=&e_3,\\[.2cm] 
e_3^2&=&e_3. 
\end{array}
$$

In this natural basis we have $D^1(1)=\{ 2 \}$, and $D^1(2)=D^1(3)=\{ 3 \}$. Then $\Pi_B(\Lambda)=\{\{1\}, \{2,3\} \}$. Now consider a new natural basis given by $B'=\{f_1,f_2,f_3\}$ with $f_1=e_1$, $f_2=e_2-e_3$ and $f_3=e_2+e_3$. In the basis $B'$ we have $D^1(1)=D^1(2)=D^1(3)=\{ 2,3 \}$ so $\Pi_{B'}(\Lambda)=\{\{1,2,3\}\}$. In Figure \ref{FIG:EA2} we represent the associated graphs for this evolution algebra.

\end{exa}

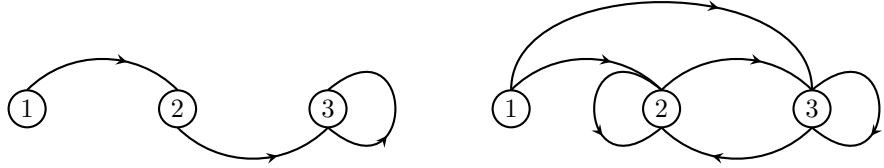
\begin{figure}[h!]
    \centering
    
    \subfigure[][Associated directed graph of the evolution algebra relative to $B$]{    \begin{tikzpicture}

\draw [thick] (0,0) circle (7pt);
\draw (0,0) node[font=\footnotesize] {1};
\draw [thick] (2,0) circle (7pt);
\draw (2,0) node[font=\footnotesize] {2};
\draw [thick] (4,0) circle (7pt);
\draw (4,0) node[font=\footnotesize] {3};


\draw [thick, directed] (0,0.25) to [bend left=45] (2,0.25);
\draw [thick, directed] (2,-0.25) to [bend left=315] (4,-0.25);
\draw [thick, reverse directed] (4,0.25) to [out=40,in=320,looseness=8] (4,-0.25);

\end{tikzpicture}}\qquad\subfigure[][Associated directed graph of the evolution algebra relative to $B'$]{    \begin{tikzpicture}

\draw [thick] (0,0) circle (7pt);
\draw (0,0) node[font=\footnotesize] {1};
\draw [thick] (2,0) circle (7pt);
\draw (2,0) node[font=\footnotesize] {2};
\draw [thick] (4,0) circle (7pt);
\draw (4,0) node[font=\footnotesize] {3};


\draw [thick, directed] (0,0.25) to [bend left=45] (2,0.25);
\draw [thick, directed] (0,0.25) to [bend left=90] (4,0.25);
\draw [thick, directed] (2,0.25) to [bend left=45] (4,0.25);
\draw [thick, directed] (2,0.25) to [out=140,in=220,looseness=8] (2,-0.25);
\draw [thick, directed] (4,-0.25) to [bend left=45] (2,-0.25);
\draw [thick, directed] (4,0.25) to [out=40,in=320,looseness=8] (4,-0.25);

\end{tikzpicture}}
    
    \caption{Associated directed graphs of the evolution algebra of Example \ref{EXA:EA2} according to the chosen basis.}
    \label{FIG:EA2}
\end{figure}

  \subsection{Characterization of the space of derivations}
We focus on non-degenerate irreducible evolution algebras with finite dimension. In other words, we consider evolution algebras such that their associated graph is finite, connected and without sinks. As mentioned before, our aim is to study the space of derivations. In the rest of the paper we shall assume that $\A$ is a $n$-dimensional $\mathbb{K}$-evolution algebra with $\chart(\K)=0$. We denote the associated graph to $\A$ by $E=(E^0,E^1,r,s)$, so $\vert E^0 \vert=n $. We suppose that $d \in \Der(\A)$ is such that 
\begin{equation}\label{eq:derexp}
d(e_i)=\sum_{k=1}^{n}d_{ik}e_k,
\end{equation} 

\noindent
with $d_{ik} \in \K$ for every $i,k \in \Lambda$. 

\begin{prop}\label{prop:1}Let $\A$ be an $n$-dimensional evolution algebra with a natural basis $B=\{e_i\}_{i\in \Lambda}$ and structure matrix $(\omega_{ij})$. If $d\in \Der\left(\A\right)$, then $d$ satisfies the following conditions:
\begin{enumerate}[label=(\roman*)]
\item If $i,j\in \Lambda$, $i\neq j$, and $ D^1(i) \cap D^1(j) \neq \emptyset$ then 
$$d_{ij}=-\frac{\omega_{ik}}{\omega_{jk}}d_{ji},  \text{ for }  k \in D^1(i) \cap D^1(j).$$
\item If $i,j\in \Lambda$, $i\neq j$, and $ D^1(i) \cap  \left(D^1(j) \right)^c \neq \emptyset$ then $d_{ji}=0$.
\item If $\A$ is non-degenerate then for every $i,j\in \Lambda$, $i\neq j$ with $D^1(i) \cap D^1(j)  = \emptyset $ we have $d_{ij}=d_{ji}=0$.
\item For any $i\in \Lambda$, we have
$$
\sum_{k\in D^1(i)} \omega_{ik}d_{kj}=\left\{
\begin{array}{cl}
0,&\text{ if }j\notin D^1(i),\\[.2cm]
2\omega_{ij}d_{ii},&\text{ if }j\in  D^{1}(i).
\end{array}\right.
$$ 
\end{enumerate}

\end{prop}

\begin{proof}
In order to prove (i) consider $i,j \in \Lambda$, with $i\neq j$, such that $D^1(i) \cap D^1(j)  \neq \emptyset$, and let $k\in D^1(i) \cap D^1(j) $. Then $\omega_{ik}\not=0$ and   $\omega_{jk}\not=0$. Therefore condition \eqref{eq:der1} implies
$$d_{ij}=-\frac{\omega_{ik}}{\omega_{jk}}d_{ji}.$$
Analogously, assume now that for some $i,j\in \Lambda$ we have $D^1(i) \cap  \left(D^1(j) \right)^c  \neq \emptyset$, and let $k\in D^1(i) \cap  \left(D^1(j) \right)^c $. 
This means that $\omega_{ik}\not=0$ while $\omega_{jk}=0$. This implies, by \eqref{eq:der1}, that $d_{ji}=0$ and the proof of (ii) is completed. Item (iii) is a consequence of (ii) and $\A$ is non-degenerate since $D^1(i) \cap D^1(j) =\emptyset$ implies $ D^1(i) \cap  \left(D^1(j) \right)^c\neq \emptyset$ and  $ D^1(j) \cap  \left(D^1(i) \right)^c\neq \emptyset$. Finally, item (iv) may be obtained by observing in \eqref{eq:der2} that for any $i\in \Lambda$, $\omega_{ik}\not=0$ if and only if,  $k\in D^1(i).$

\end{proof}

\begin{corollary}
Let $\A$ be an $n$-dimensional evolution algebra with a natural basis $B=\{e_i\}_{i\in \Lambda}$. Let $d\in \Der\left(\A\right)$. If $d_{ij}=0$ and $i \sim_{t_B}j$ then $d_{ji}=0$.
\end{corollary}

\begin{proof} \label{dij0}
The proof is straightforward from Proposition \ref{prop:1}(i).
\end{proof}

\smallskip

\begin{lemma}\label{lema:1} Let $\A$ be a non-degenerate $n$-dimensional evolution algebra with a natural basis $B=\{e_i\}_{i\in \Lambda}$ and structure matrix $M_B=(\omega_{ij})$. Let $d\in \Der\left(\A\right)$. If $d_{ij}=d_{ji}=0$ for any $i,j\in \Lambda$, $i\neq j$, then $d_{ii}=0$ for any $i\in \Lambda$. 
\end{lemma}

\begin{proof}
Let $i\in \Lambda$ such that $d_{ii}\neq 0$.  If $j \in D^1(i)$ applying Proposition \ref{prop:1}(iv) we have that  $$ 2\omega_{ij}d_{ii}  = \sum_{k\in D^1(i)} \omega_{ik}d_{k j} = \omega_{ij}d_{jj}.$$ Therefore 
\begin{equation}\label{1eq:lem1}
2d_{ii}=d_{jj}, \text{ for } j \in D^1(i) .
\end{equation}

Now, we will see that a directed graph $E$ that has no sinks contains at least one cycle. Indeed, let $u_1 \in E^0$ an arbitrary vertex. If there exists an arrow $g \in E^1$ such that $s(g)=r(g)=u_1$ then we get the desired result. On the contrary, since $u_1$ is not a sink, there exist a vertex $u_2 \in E^0 \setminus\{u_1\}$ and arrow $f_1$ such that $s(f_1)=u_1$ and $r(f_1)=u_2$. Now, reasoning in the same way with $u_2$, we get that there exist either a cycle of length 1, a cycle of length 2 or there exist a vertex $u_3 \in E^0\setminus\{u_1, u_2\}$ and an arrow $f_2\in E^1$ such that $s(f_2)=u_2$ and $r(f_2)=u_3$. Since the algebra has dimension n, we iterate this process at most $(n-1)$ times and either we have a cycle or we have a path $f_1f_2\ldots f_{n-1}$ that contains n different vertices $\{u_1, \ldots ,u_n\}$. Since $u_n$ is not a sink, necessarily there exist a vertex $u_i \in E^0$ and an arrow $f_n$ such that $s(f_n)=u_n$ and $r(f_n)=u_i$. Therefore, in any case, we get a cycle. 

Let $c_1, \ldots, c_s$ be the cycles of $E$ and $c_1^0, \ldots, c_s^0 $ the set of vertices of such  cycles, that is $c_i^0=\{v_{i_1},\ldots, v_{i_{t_i}}\}$ with $v_{i_j} \in E^0$ for every $i \in \{1,\ldots,s\}$ and $j \in \{1,\ldots,t_i\}$. For $j \in \{1,\ldots,(t_{i}-1)\}$ and $i \in \{ 1,\ldots, s\}$ we have that ${i_{j+1}} \in D^1({i_j})$ and $i_{1} \in  D^1{(i_{t_i})}$  then 
by \eqref{1eq:lem1} 
 $$2d_{i_ji_j}=d_{i_{j+1}i_{j+1}}\,\,\,\text{ and }\,\,\, 2d_{i_ji_j}=d_{i_{j+1}i_{j+1}}.$$
 Therefore, $d_{i_ji_j}=0$ for all $j \in \{1, \ldots,t_i\}$ and $i \in \{1,\ldots,s \}$. Suppose that there exists a vertex $v_k$ such that $v_k \in E^0 \setminus (\cup_{i=1}^s c_i^0)$. Since $v_k$ is not a sink there exist an $i_0 \in \{1,\ldots,s\}$ and $h \in \{1,\ldots,t_{i_0}\}$ with $v_{i_{0_h}}\in c^0_{i_0}$ and a path $\mu$ with $\mu^0=\{v_k,v_{k_1},\ldots,v_{k_l},v_{i_{0_h}}\}$ such that $v_{k_i} \in E^0 \setminus (\cup_{i=1}^s c_i^0)$  for  all $i\in \{1,\ldots,l \}$. Since $d_{i_{0_h}i_{0_h}}=0$ then $d_{k_ik_i}=0$ for  all $i\in \{1,\ldots,l \}$ and therefore $d_{kk}=0$.


\end{proof}

\begin{lemma} \label{lema:2} Let $\A$ be a non-degenerate $n$-dimensional evolution algebra with a natural basis $B=\{e_i\}_{i\in \Lambda}$ and structure matrix $M_B=(\omega_{ij})$. Let $d\in \Der\left(\A\right)$. If $d_{ij}\not=0$ for $i,j \in \Lambda$, $i\not=j$ then $i \sim_{t_B}j$.
\end{lemma}
\begin{proof} 
Let $i,j \in V$, $i\not=j$ such that $d_{ij}\not=0$. By Proposition \ref{prop:1} (iii) we have that $ D^1(i) \cap D^1(j) \not=\emptyset $ then by Proposition \ref{prop:1} (i) we have that 
$$d_{ji}=-{\frac{\omega_{jk}}{\omega_{ik}}} d_{ij} \neq 0  \text{ for } k\in  D^1(i) \cap D^1(j). $$
Then, by Proposition \ref{prop:1} (ii), 
$$D^1(i) \cap  \left(D^1(j) \right)^c  = \emptyset  \text{ and } D^1(j) \cap  \left(D^{1}(i) \right)^c  = \emptyset . $$
Therefore, $ D^1(i)= D^1(j)$.
\end{proof}

\smallskip
\begin{theorem}\label{thm:main} Let $\A$ be an irreducible non-degenerate $n$-dimensional evolution algebra. Let $B=\{e_i\}_{i\in \Lambda}$ be a natural basis of $\A$ with structure matrix $M_B=(\omega_{ij})$, and let $d\in \Der\left(\A\right)$. If $\A$ is twin-free relative to $B$, then $d=0$.
\end{theorem}

\begin{proof}
Let $d\in \Der\left(\A\right)$. Since $\A$ is twin-free relative to $B$ we have that $i\not \sim_{B} j$ for any $i,j\in \Lambda$, with $i\neq j$. This in turns implies, by Lemma \ref{lema:2}, $d_{ij}=d_{ji}=0$, for any $i,j\in \Lambda$, with $i\neq j$. Therefore, by Lemma \ref{lema:1} we conclude $d=0$. 


\end{proof}

The previous theorem give us a first taste of how strongly the associated graph's structure impacts in the existence, or not, of non-zero derivations for a given evolution algebra. The theorem gains in interest if we realize that it provides a simple criterion for the identification of evolution algebras for which the only derivation is the null map. Moreover, it provides an alternative to previous analysis related in the literature about this topic. For the sake of comparison we illustrate the applicability of Theorem \ref{thm:main} with some examples.

\begin{exa}\label{EXA:EA3}
\rm
Consider the evolution algebra $\A$ with natural basis $B=\{e_1,e_2,e_3,e_4,e_5\}$ such that

$$
\begin{array}{ccl}
 e_1^2&=& e_1 + 2\, e_2 + e_4 - e_5,\\[.2cm]
 e_2^2&=& -3\, e_1 + 3 e_3 + 2\, e_5,\\[.2cm]
 e_3^2&=& e_2 + e_4 - 3\, e_5,\\[.2cm]
 e_4^2&=& -2\, e_3 + e_1 + e_4 + e_5,\\[.2cm]
  e_5^2&=& - e_1 + e_2 + 2\, e_3 - 5\, e_4,\\[.2cm]
\end{array}
$$

It is not difficult to see that the associated graph (see Figure \ref{FIG:wheel}(a)) is twin-free relative to $B$. In this case, $D^1(1)=\{1,2,4,5\}$, $D^1(2)=\{1,3,5\}$, $D^1(3)=\{2,4,5\}$, $D^1(4)=\{1,3,4,5\}$ and $D^1(5)=\{1,2,3,4\}$. Therefore, Theorem \ref{thm:main} implies that if $d\in \Der\left(\A\right)$ then $d=0$. Notice that the same result may be obtained as a consequence of \cite[Theorem 2.1]{COT} or \cite[Theorem 4.1(1)]{Elduque/Labra/2019} since the structure matrix is a non-singular matrix. 
\end{exa}

\begin{figure}
    \centering
    
    \subfigure[][Associated graph for the evolution algebra of Example \ref{EXA:EA3}]{

\begin{tikzpicture}[scale=0.93]

\draw [thick] (0,0) circle (7pt);
\draw (0,0) node[font=\footnotesize] {$5$};
\draw [thick] (2,2) circle (7pt);
\draw (2,2) node[font=\footnotesize] {$1$};
\draw [thick, directed] (2,2.25) to [out=90,in=0,looseness=8] (2.25,2);
\draw [thick] (2,-2) circle (7pt);
\draw (2,-2) node[font=\footnotesize] {$2$};
\draw [thick] (-2,-2) circle (7pt);
\draw (-2,-2) node[font=\footnotesize] {$3$};
\draw [thick] (-2,2) circle (7pt);
\draw (-2,2) node[font=\footnotesize] {$4$};
\draw [thick, directed] (-2.25,2) to [out=180,in=90,looseness=8] (-2,2.25);


\draw [thick, directed] (-2,2.25) to [bend left=20] (2,2.25);
\draw [thick, reverse directed] (-1.75,2) to  (1.75,2);

\draw [thick, directed] (-1.75,-2) to (1.75,-2);
\draw [thick, reverse directed] (-2,-2.25) to [bend left=340] (2,-2.25);

\draw [thick, directed] (2.25,2) to [bend left=20] (2.25,-2);
\draw [thick, reverse directed] (2,1.75) to  (2,-1.75);

\draw [thick, directed] (-2.25,2) to [bend left=340] (-2.25,-2);
\draw [thick, reverse directed] (-2,1.75) to  (-2,-1.75);

\draw [thick, directed] (0,0.25) to [bend left=20] (1.75,2);
\draw [thick, directed] (2,1.75) to [bend left=20] (0.25,0);

\draw [thick, directed] (0.25,0) to [bend left=20] (2,-1.75);
\draw [thick, directed] (1.75,-2) to [bend left=20] (0,-0.25);

\draw [thick, directed] (0,-0.25) to [bend left=20] (-1.75,-2);
\draw [thick, directed] (-2,-1.75) to [bend left=20] (-0.25,0);

\draw [thick, directed] (-0.25,0) to [bend left=20] (-2,1.75);
\draw [thick, directed] (-1.75,2) to [bend left=20] (0,0.25);


\end{tikzpicture}}\qquad \qquad\subfigure[][Associated graph for the evolution algebra of Example \ref{EXA:EA4}.]{
\begin{tikzpicture}[scale=0.93]
\draw [thick] (0,0) circle (7pt);
\draw (0,0) node[font=\footnotesize] {$1$};
\draw [thick] (2,0) circle (7pt);
\draw (2,0) node[font=\footnotesize] {$2$};
\draw [thick] (4,0) circle (7pt);
\draw (4,0) node[font=\footnotesize] {$3$};
\draw [thick] (6,0) circle (7pt);
\draw (6,0) node[font=\footnotesize] {$4$};

\draw [thick, directed] (0.25,0) to (1.75,0);
\draw [thick, directed] (2.25,0) to [bend left=30] (5.75,0);
\draw [thick, directed] (3.75,0) to (2.25,0);
\draw [thick, directed] (4.25,0) to (5.75,0);
\draw [thick, directed] (4,-0.25) to [bend left=30] (0,-0.25);
\draw [thick, directed] (6,0.25) to [bend left=300] (0,0.25);

\draw [thick, directed] (1.75,0) to [out=130,in=50,looseness=8] (2.25,0);

\draw [thick, directed] (0,0.25) to [out=140,in=220,looseness=8] (0,-0.25);
\draw [thick, directed] (6,0.25) to [out=40,in=320,looseness=8] (6,-0.25);

\draw (0,-2.5) node[white, font=\footnotesize] {$4$};

\end{tikzpicture}
}

    \caption{Associated twin-free graphs of evolution algebras.}
    \label{FIG:wheel}
\end{figure}
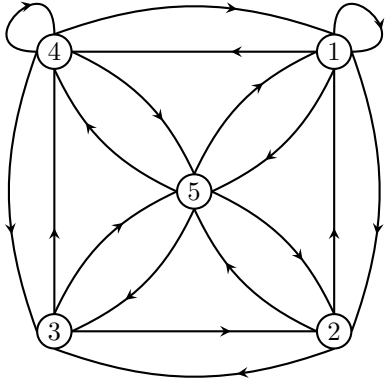
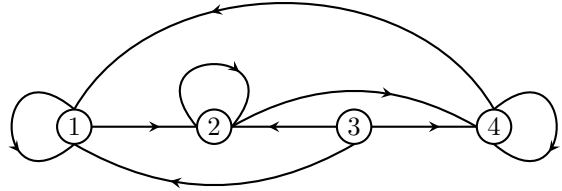

\begin{exa}\label{EXA:EA4} \rm
\rm
Consider the evolution algebra $\A$ with natural basis $B=\{e_1,e_2,e_3,e_4\}$ such that 
$$
\begin{array}{ccl}
e_1^2&=&e_1 + e_2,\\[.2cm]
e_2^2&=&e_2 + e_4,\\[.2cm]
e_3^2&=&e_1 + 2e_2 + e_4,\\[.2cm]
e_4^2&=&e_1 - e_4.
\end{array}
$$

Note that $\A$ is twin-free relative to $B$ since $D^1(1)=\{1,2\}$, $D^1(2)=\{2,4\}$, $D^1(3)=\{1,2,4\}$ and $D^1(4)=\{1,4\}$. Therefore, Theorem \ref{thm:main} implies that $d=0$ for all $d \in \Der\left(\A\right)$. Contrary to what happens in Example \ref{EXA:EA3} we cannot apply results from \cite{COT,Elduque/Labra/2019} because this case is an example of an $n$-dimensional evolution algebra whose structure matrix has rank $n-2$; which is a case not covered by those works. Observe that since $\A$ is a $4$-dimensional evolution algebra, it is not nilpotent, and since there exists $\omega_{ij}\not \in\{0,1\}$ we cannot apply the results from \cite{Alsarayreh/Qaralleh/Ahmad/2017,PMP3,cardoso,Mukhamedov/Khakimov/Omirov/Qaralleh/2019} neither. 
\end{exa}

We emphasize that, given an evolution algebra, the twin partition of its associated graph plays an important role in the description of its space of derivations. Our approach extends the one developed by \cite{PMP3} for the particular class of evolution algebras associated to graphs. Our next task is to explore derivations of those evolution algebras with at least one twin class with more than one vertex.

\smallskip
\begin{theorem}\label{thm:main2}
Let $\A$ be a non-degenerate $n$-dimensional evolution algebra. Let $B=\{e_i\}_{i\in \Lambda}$ be a natural basis of $\A$ with structure matrix $M_B=(\omega_{ij})_{i,j\in \Lambda}$, and let $\Pi_B(\Lambda)=\{\mathcal{T}_1,\mathcal{T}_2,\ldots,\mathcal{T}_k\}$ be the twin partition of $\Lambda$ relative to $B$, with $k\in\{1,\ldots,n\}$. If $d\in \Der\left(\A\right)$, then
$$d=A_1 \oplus A_2 \oplus \cdots \oplus A_k,$$
where $A_i$ is a $|\mathcal{T}_i|$-dimensional square matrix, for $i\in \{1,\ldots,k\}$.
\end{theorem}

\begin{proof}
It is enough to show that $d_{ij}=0$ provided $i,j$ belong to different twin classes. This is a consequence of Lemma \ref{lema:2} because by taking $i,j\in \Lambda$, $i\neq j$, such that $i\not \sim_{B} j$ then $d_{ij}=d_{ji}=0$. Thus one can identify each matrix with a twin class after a suitable ordering of the generators index set.


\end{proof}

The principal significance of Theorem \ref{thm:main2} is in the assertion that the derivations of an evolution algebra are strongly related to the twin partition of its associated directed graph. It claims that the induced matrix of a derivation is the direct sum of matrices induced by the twin classes of the respective graph. This is one step in the direction of a generalized version of Theorem 2.6 in \cite{PMP3}. However, it is worth pointing out that, while in \cite{PMP3} the twin classes of at most two vertices induce only null matrices, here we could have a non-zero matrix associated to such twin classes.

\begin{exa} \rm
Let $\A$ be an evolution algebra with natural basis $B=\{e_1,e_2,e_3\}$ such that 
$$
\begin{array}{ccl}
e_1^2&=&e_3,\\[.2cm] 
e_2^2&=&-e_3,\\[.2cm]
e_3^2&=&e_1+e_2.
\end{array}
$$

In this natural basis we have that
$D^1(1)=D^1(2)=\{3 \}$ and $D^1(3)=\{ 1,2 \}$. Then $\Pi_B(\Lambda)=\{\T_1,\T_2 \}$, where $\T_1=\{1,2\}$, $\T_2=\{3\}$, and it is not difficult to see that
$$d=\begin{pmatrix} 1 & 3 & 0 \\ 3 & 1 & 0 \\
0 & 0 & 2  \end{pmatrix}$$ 
belongs to $\Der\left(\A\right)$. Therefore this is an example of evolution algebra with twin classes of size at most two and such that $d=A_1 \oplus A_2 $ with $A_1\neq 0$ and $A_2\neq 0$.
\end{exa}

From now on we shall investigate the behavior of the  matrices appearing in Theorem \ref{thm:main2}. 

\subsection{The null entries of a derivation matrix when restricted to a twin class}\label{section:null}

As we shall see later, to know that some elements in the matrix representation of a given derivation are equal to zero allow us to have some information about the remaining elements. Let us start with some general conditions for a given pair $i,j\in \Lambda$ to guarantee $d_{ij}=0$. Then we restrict our attention to those elements inside a given twin class.

\begin{prop}\label{prop:general1}
Let $\A$ be a $n$-dimensional evolution algebra with a natural basis $B=\{e_i\}_{i\in \Lambda}$ and structure matrix $M_B=(\omega_{ij})$. Let $d\in \Der\left(\A\right)$ and consider $i,j\in \Lambda$, with $i\neq j$. Assume that $|D^1(i)|\geq 2$ and that there exist $k, \ell \in D^1(i)$ such that
$$\det
\begin{pmatrix}
\omega_{ik} & \omega_{jk}\\
\omega_{i\ell} & \omega_{j\ell}
\end{pmatrix} \neq 0
.$$

Then $d_{ij}=d_{ji}=0$.
\end{prop}

\begin{proof}
Assume $|D^1(i)|\geq 2$ and note that $k, \ell \in D^1(i)$ implies by Proposition \ref{prop:1}(i) that
$$d_{ij}=-\frac{\omega_{ik}}{\omega_{jk}}d_{ji}, \text{ and }d_{ij}=-\frac{\omega_{i\ell}}{\omega_{j\ell}}d_{ji}.$$
This in turns implies 
$$(\omega_{ik}\omega_{j\ell}-\omega_{i\ell}\omega_{jk})d_{ji}=0,$$
so by hypothesis we obtain $d_{ji}=0$ and again by Proposition \ref{prop:1}(i) we have $d_{ij}=0$. 
\end{proof}

\begin{prop}\label{prop:general0}
Let $\A$ be a $n$-dimensional evolution algebra with $n \geq 3$. Let $B=\{e_i\}_{i\in \Lambda}$ be a natural basis of $\A$ with structure matrix $M_B=(\omega_{ij})$, and let $d\in \Der\left(\A\right)$. Consider $i,j\in \Lambda$, with $i\neq j$. If 
\begin{equation}\label{ceros}
\left(\sum_{k\in\Lambda}\omega_{ik}^2\right) \left(\sum_{k\in\Lambda}\omega_{jk}^2\right) - \left(\sum_{k\in\Lambda}\omega_{ik}\omega_{jk}\right)^2 \neq 0,
\end{equation}
then $d_{ij}=d_{ji}=0$.
\end{prop}

\begin{proof}Fix $i,j\in \Lambda$, with $i\neq j$. If we consider the homogeneous system $\omega_{jk}d_{ij}+\omega_{ik}d_{ji}=0$, for $k\in \Lambda$, with variables $d_{ij}$ then the system is an over determined system. By applying the least-squares solution we get the homogeneous system
$$(d_{ji} d_{ij}) \begin{pmatrix}
 \displaystyle\sum_{k\in\Lambda}\omega_{ik}^2    & \displaystyle\sum_{k\in\Lambda}\omega_{ik}\omega_{jk} \\[.4cm]
\displaystyle\sum_{k\in\Lambda}\omega_{ik}\omega_{jk}     & \displaystyle\sum_{k\in\Lambda}\omega_{jk}^2
\end{pmatrix}=0$$
which has trivial solution provided its determinant is non-zero.
\end{proof}

\begin{remark} \rm
Proposition \ref{prop:general0} holds for evolution algebras of dimension three or higher. In the case of a two-dimensional evolution algebra, the above proposition means as follows if the evolution algebra is perfect, i.e., $A=A^2$, then $d_{12}=d_{21}=0$ being the natural basis $B=\{e_1,e_2\}$. 
\end{remark}

\begin{exa} \rm
Let $\A$ be an evolution algebra and a natural basis $B$ with product
$$
\begin{array}{ccl}
e_1^2&=&(1/2)e_1 -(1/4)e_2,\\[.2cm] 
e_2^2&=&-2e_1+e_2,\\[.2cm]
e_3^2&=&2e_1+e_2.
\end{array}
$$
Note that we have a unique twin class, that is, $\Pi_B(\Lambda)=\{\T_1\}$, where $\T_1=\{1,2,3\}$. The proposition above implies that $d_{13}=d_{31}=d_{23}=d_{32}=0$, which simplifies the task of calculating the derivations. Indeed, a straightforward calculation shows that the derivations are given by
$$\begin{pmatrix}
d_{11} & -(1/2)d_{11} & 0 \\
-2d_{11} & d_{11} & 0 \\
0 & 0 & 0
\end{pmatrix}.$$
\end{exa}

Notice that in the previous example we have a unique twin class. Moreover, observe that in this example $i\in D^1(i)$ whether $i\in \{1,2\}$, and we can write $d=A_1 = B_1 \oplus B_2$ where the $B_i$ matrices are related to those vertices with and without loops in the associated directed graph, respectively. This suggest that a good strategy to understand the behavior of the matrices $A_i$ in Theorem \ref{thm:main2} could be focusing in the existence of loops in each directed graph associated to each twin partition. 

\smallskip

\begin{prop}
Let $\A$ be an $n$-dimensional evolution algebra with a natural basis $B=\{e_i\}_{i\in \Lambda}$. Let $\T$ be a twin class relative to $B$ of $\Lambda$ such that $\T=\{i\},$ for certain $i \in \Lambda$. If $i\in D^1(i)$ then $d_{ii}=0.$ 
\end{prop}

\begin{proof}
This is a direct consequence of Proposition \ref{prop:1}(iv) and Lemma \ref{lema:2}. Indeed, if we take $j=i$ we obtain, as $i\in D^1(i)$, the following expression:
$$\omega_{ii}d_{ii} + \sum_{k\in D^1(i)\cap \T^c}\omega_{ik}d_{ki} = 2\omega_{ii}d_{ii},$$
but, Lemma \ref{lema:2} implies $d_{ki}=0$ for any $k\in \T^c$. Therefore, $d_{ii}=0$.  
\end{proof}

\smallskip
From now on our approach is to subdivide a twin class into two disjoint sets formed by those vertices with a loop (we use $wl$ in the notation) and those vertices without a loop (we use $nl$ in our notation). 

\begin{prop}\label{prop:loops1}
Let $\A$ be an $n$-dimensional evolution algebra with a natural basis $B=\{e_i\}_{i\in \Lambda}$. Let $\T$ be a twin class relative to $B$ of $\Lambda$. Assume that $\mathcal{T}=\mathcal{T}^{wl}\cup \mathcal{T}^{nl}$, where 
$$\mathcal{T}^{wl}:=\{i\in \mathcal{T}: i\in D^{1}(i)\}\neq \emptyset,$$
and $\mathcal{T}^{nl}:= \mathcal{T} \setminus \mathcal{T}^{wl}$. Also, assume that $d_{ij}=0$ for any $i,j\in \mathcal{T}^{wl}, i\neq j$. Then,

\begin{enumerate}
\item[(i)] $d_{ii}=0$ for any $i\in \mathcal{T}$.
\item[(ii)] if $j\in \mathcal{T}^{nl}$ and $i\in \mathcal{T}^{wl}$ we have
$$\sum_{k\in \mathcal{T}^{wl}} w_{ik} d_{kj}=0.$$

\noindent
In other words, if $W^{(wl)}:=(w_{ij})_{i,j \in \mathcal{T}^{wl}}$ and $\tilde{D}:=(d_{ij})_{i\in \mathcal{T}^{wl}, j\in \mathcal{T}^{nl}}$ then $W^{(wl)} \tilde{D} = 0$.
\item[(iii)] If $W^{(wl)}=(w_{ij})_{i,j \in \mathcal{T}^{wl}}$ is a non-singular matrix, then $d_{ij}=0$ for any $i,j$ such that $i\in \T^{wl}$ or $j\in \T^{wl}$.

\end{enumerate}
\end{prop}

\begin{proof}

Let us prove $(i)$. We consider two separate cases.\\ 

\noindent
{Let $i\in \mathcal{T}^{wl}$.} Then, Proposition \ref{prop:1} (iv) (by assuming $j=i$) implies that
$$2w_{ii}d_{ii} = \sum_{k\in D^{1}(i)} w_{ik} d_{ki} = \sum_{k\in D^{1}(i)\cap \mathcal{T}} w_{ik} d_{ki} + \sum_{k\in D^{1}(i)\cap \mathcal{T}^c} w_{ik} d_{ki}.$$
But 
$$\sum_{k\in D^{1}(i)\cap \mathcal{T}} w_{ik} d_{ki} = \sum_{k\in \mathcal{T}^{wl}} w_{ik} d_{ki},$$

\noindent
while 

$$\sum_{k\in D^{1}(i)\cap \mathcal{T}^c} w_{ik} d_{ki} =0,$$

\noindent
because $k\in D^{1}(i)\cap \mathcal{T}^{c}$ and $i$ belong to different twin classes, so $d_{ki}=0$. Then, we conclude that 

$$2w_{ii}d_{ii} =\sum_{k\in \mathcal{T}^{wl}} w_{ik} d_{ki}=w_{ii}d_{ii},$$

\noindent
because we are assuming $d_{ki}=0$ whether $k,i\in  \mathcal{T}^{wl}, k\neq i$. Then $d_{ii}=0$.\\

\noindent
{ Let $i\in \mathcal{T}^{nl}$.} Proposition \ref{prop:1} (iv) (by assuming $j\in \mathcal{T}^{wl}$) implies that 
$$2w_{ij}d_{ii} = \sum_{k\in D^{1}(i)} w_{ik} d_{kj} = \sum_{k\in D^{1}(i)\cap \mathcal{T}} w_{ik} d_{kj} + \sum_{k\in D^{1}(i)\cap \mathcal{T}^c} w_{ik} d_{kj}.$$

\noindent
In this case, we obtain reasoning as before

$$2w_{ij}d_{ii} = \sum_{k\in \mathcal{T}^{wl}} w_{ik} d_{kj} = w_{ij} d_{jj} + \sum_{k\in \mathcal{T}^{wl}, k\neq j} w_{ik} d_{kj}= w_{ij} d_{jj} =0,$$

\noindent
where $d_{jj}=0$ because of the previous case. So $d_{ii}=0.$\\

\noindent
Now let us consider $(ii)$. Take $i\in \mathcal{T}^{wl}$ and $j\in \mathcal{T}^{nl}$. We apply Proposition \ref{prop:1} (iv) to obtain

$$ \sum_{k\in D^{1}(i)} w_{ik} d_{kj} = 2 w_{ij} d_{ii} =0.$$

\noindent
Then

$$0=  \sum_{k\in D^{1}(i)} w_{ik} d_{kj} =  \sum_{k\in D^{1}(i)\cap \mathcal{T}} w_{ik} d_{kj} + \sum_{k\in D^{1}(i)\cap \mathcal{T}^c} w_{ik} d_{kj} =  \sum_{k\in  \mathcal{T}^{wl}} w_{ik} d_{kj},$$

\noindent
and the proof is complete. We note that condition $(ii)$ may be written in matricial form. The result is the statement of condition $(iii)$. 

\end{proof}

\begin{exa} \rm
Let $\A$ an evolution algebra with product 
$$
\begin{array}{ccl}
e_1^2&=&e_1-e_2+e_3,\\[.2cm]
e_2^2&=&e_1+e_2-e_3,\\[.2cm] 
e_3^2&=&e_1+e_2-e_3,\\[.2cm]
e_4^2&=&-e_1-e_2+e_3,
\end{array}
$$
and let $d\in Der(\A)$. We point out that the structure matrix has range $2$ so it is not covered by \cite{COT,Elduque/Labra/2019}. Note that this case is not covered by \cite{Alsarayreh/Qaralleh/Ahmad/2017,PMP2,cardoso,Mukhamedov/Khakimov/Omirov/Qaralleh/2019} neither. Here we have $\Pi_B(\Lambda)=\T$, where $\T=\{1,2,3,4\}$. Moreover, $\T^{wl}=\{1,2,3\}$ and $\T^{nl}=\{4\}$, and we can apply Proposition \ref{prop:general0} to conclude that $d_{14}=d_{41}=0$ and $d_{ij}=d_{ji}=0$ provided $i,j\in \T^{wl}$, $i\neq j$. Then Proposition \ref{prop:loops1}(i) implies $d_{ii}=0$ for any $i\in \T$. Finally, Proposition \ref{prop:loops1}(ii) and \eqref{eq:der1} implies that $d$ is equal to

$$\begin{pmatrix}
0 & 0 & 0 & 0 \\
0 & 0 & 0 & d_{24} \\
0 & 0 & 0 & d_{24} \\
0 & d_{24} & d_{24}  & 0 \\
\end{pmatrix}.$$
\end{exa}


\begin{exa} \rm
Consider the evolution algebra $\A$ with natural basis $B=\{e_1,e_2,e_3,e_4,e_5\}$ such that

$$
\begin{array}{ccl}
 e_1^2&=& e_1 + 2\, e_2 + 3\, e_3,\\[.2cm]
 e_2^2&=&  e_1 + e_2 +  3\, e_3,\\[.2cm]
 e_3^2&=& 2\, e_1 + e_2 + e_3,\\[.2cm]
 e_4^2&=& -2\, e_1 + e_2 - e_3,\\[.2cm]
  e_5^2&=& e_1 + e_2 + e_3.\\[.2cm]
\end{array}
$$
This is another example of $n$-dimensional evolution algebra whose structure matrix has rank equals to $n-2$. Therefore, the results in \cite{COT,Elduque/Labra/2019} does not apply, and the results in \cite{PMP3,cardoso} neither. Here we have $\Pi_B(\Lambda)=\T$ with $\T^{wl}=\{1,2,3\}$ and $\T^{nl}=\{4,5\}$. By Proposition \ref{prop:general0} we have $d_{ij}=0$ provided $i,j\in \T^{wl}$, $i\neq j$. In addition, by Proposition \ref{prop:loops1}(iii) we conclude that $d_{ij}=0$ provided $i$ or $j$ belongs to $\T^{wl}$. In fact, note that 
$$
W^{(wl)}=\begin{pmatrix}
1&2&3\\
1&1&3\\
2&1&1
\end{pmatrix},
$$
is a non-singular matrix. Finally, Proposition \ref{prop:general1} allows to conclude that $d_{ij}=0$ for $i,j\in \T^{nl}$ and therefore if $d\in \Der(\A)$ then $d=0$.

\end{exa}

\begin{lemma}
Let $\A$ be a non-degenerate $n$-evolution algebra with a natural basis $B=\{e_i\}_{i \in \Lambda}$ and structure matrix $M_B=(\omega_{ij})$. Suppose that there exists a twin class $\T$ such that $\T^{wl}=\emptyset$. Let $i \in \T$ and suppose that for every $j_1, \, j_2 \in D^1(i)$, $\T_{j_1}\neq \T_{j_2}$. Moreover, we assume that $d_{jj}=0$ for certain $j \in D^1(i)$. Then $d_{ii}=0$ for every $i \in \T$.
 
\end{lemma}

\begin{proof}
Note that  $d_{j_1j_2}=d_{j_2j_1}=0$ for every $j_1, \, j_2 \in D^1(i)$ because $j_1\nsim_{t_B} j_2$. On the other hand, by Proposition \ref{prop:1}(iv) we have that $$\displaystyle\sum_{k\in D^1(i)}\omega_{ik}d_{kj}=2\omega_{ij}d_{ii}$$ which implies that  $\omega_{ij}d_{jj}=2\omega_{ij}d_{ii}$ since $d_{kj}=0$ if $k \neq j$. Finally, as $d_{jj}=0$ we get that $d_{ii}=0$ for every $i \in \T$.
\end{proof}


\begin{lemma}\label{lema221}
Let $\A$ be an $n$-evolution algebra with a natural basis $B=\{e_i\}_{i \in \Lambda}$ and structure matrix $M_B=(\omega_{ij})$. Suppose that there exists a twin class $\T$ and an element $i \in \T$ such that $D^1(j)=\{i\}$ for every $j \in \T$. Then $d_{ij}=d_{ji}=d_{ii}=d_{jj}=0$ for every $j\in \T$.
 \end{lemma}

\begin{proof}
Let $j \in \T \setminus{\{i\}}$. Note that $j \in \T^{nl}$ and $i\in \T^{wl}$. By Proposition \ref{prop:loops1} we get $\omega_{ii}d_{ij}=0$ therefore $d_{ij}=0$. By Corollary \ref{dij0}, $d_{ji}=0$. Moreover, by \eqref{eq:der2}, $w_{ii}d_{ii}=2\omega_{ii}d_{ii}$ so $d_{ii}=0$. Again, by \eqref{eq:der2}, $\omega_{ji}d_{ii}=2\omega_{ji}d_{jj}$ then $d_{jj}=0$.

\end{proof}

\begin{lemma}\label{lema222}
Let $\A$ be an $n$-evolution algebra with a natural basis $B=\{e_i\}_{i \in \Lambda}$  and structure matrix $M_B=(\omega_{ij})$. Let $\T$ be a twin class relative to $B$ and we consider the set $D^1(\T)$. Suppose that there exists $k \in \Lambda$ such that $D^1(k) \cap \T=\{j\}$ for certain $j \in \T$. Then $d_{jl}=d_{lj}=0$ for every $l \in \T \setminus \{j\}$. 
\end{lemma}

\begin{proof}
Let $\T$ be a twin class and $k \in \Lambda$ which verifies $D^1(k) \cap \T=\{j\}$. By \eqref{eq:der2}, $\omega_{kj}d_{jl}=2\omega_{kl}d_{kk}$ for every $l \in \T$ because $d_{jl}=0$ for every $l \in \T^c$ (see Lemma \ref{lema:2}) and $D^1(k) \cap \T=\{j\}$. If  $l \in \T \setminus \{j\}$ then $\omega_{kj}d_{jl}=0$. Therefore $d_{jl}=0$. Moreover, by Lemma \ref{dij0} we get $d_{lj}=0$.
\end{proof}

\begin{lemma}\label{lema223}
Let $\A$ be an evolution algebra with a natural basis $B=\{e_i\}_{i \in \Lambda}$ and structure matrix $M_B=(\omega_{ij})$. Let $\T$ be a twin class relative to $B$ and we consider $k \in D^1(\T)$. Suppose that  $D^1(k) \subseteq \T \cup \{k\}$. Moreover, we suppose that $d_{sk}=d_{ks}=0$ for every $s \in D^1(\T)\setminus\{k\}$. If $D^1(\T)\cap \T=\{i,j\}$ for certain $i,j \in \T$ with $i,j \neq k$ then $d_{ji}=d_{ij}=d_{ii}=d_{jj}=d_{kk}=0$. 
\end{lemma}

\begin{proof}
First, we will see that $d_{kk}=0$. Note that $k \in \T^c$. By \eqref{eq:der2} and $D^1(k) \subseteq \T \cup \{k\}$  we have that $$\displaystyle\sum_{h\in \T}\omega_{kh}d_{hk}+\omega_{kk}d_{kk}=2\omega_{kk}d_{kk}. $$ 
Applying Lemma \ref{lema:2} we get that $d_{kk}=0$. If we consider the descendants of $i$ and \eqref{eq:der2} we get that 
$$\displaystyle\sum_{h\in D^1(\T)\setminus\{k\}}\omega_{ih}d_{hk}+\omega_{ik}d_{kk}=2\omega_{ik}d_{ii}$$. Since $d_{kk}=0$
this implies that $d_{ii}=0$. Similarly, $d_{jj}=0$. On the other hand, again by \eqref{eq:der2} 
$$ \omega_{ii}d_{ii}+\omega_{ij}d_{ji}+\displaystyle\sum_{h\in \T^c}\omega_{ih}d_{hi}=2\omega_{ii}d_{ii}.$$ 
But $d_{hi}=0$ for every $h \in \T^c$ applying Lemma \ref{lema:2} and $d_{ii}=0$ therefore $d_{ji}=0$ and by Lemma \ref{dij0} $d_{ij}=0$.
\end{proof}


\smallskip
\section{Derivations of non-degenerate irreducible $3$-dimensional\\ evolution algebras}

In this section we apply our results to identify the derivations of any non-degenerate and irreducible $3$-dimensional evolution algebra. In order to accomplish this task we consider such an evolution algebra, and we consider its associated directed graph. We know from \cite{Elduque/Labra/2015} that any evolution algebra with a given natural basis induces a unique directed graph. Moreover, our results show how strong is the connection between the structure of such associated graph and the derivations of the considered evolution algebra. As we state in Theorem \ref{thm:main} if the graph is twin-free then the only derivation is the null map. Thus we will focus our attention in studying those cases where the associated graph is not twin-free. Since we have only $3$ vertices, the only possibilities are that we have a twin class with three or with two elements.

The table below shows the associated graph and the derivations of the corresponding $3$-dimensional evolution algebras depending on the number of the non-zero entries in its structure matrix i.e., the number of arrows in the associated graph. For each case we identify the elements, and the descendants, in the biggest twin class. 

Section \ref{section:null} summarize many sufficient conditions to guarantee the existence of null elements in the derivation matrix of a given evolution algebra. In what follows we shall see that through a suitable application of these results we can identify many cases of evolution algebras which the only derivation is the null map. 

Indeed, those evolution algebras of type $2, 4, 5$ and 6  have derivations zero as a consequence of Lemma \ref{lema221}, Lemma \ref{lema:2}, Lemma \ref{dij0} and Lemma \ref{lema:1}, in that order. On the other hand, the evolution algebras associated to the types $17$ and $21$ have derivations zero by Lemma \ref{lema223}. Also, notice that any evolution algebra of type $3,$ $8$ to $12$, $14, 15, 16, 18, 20$ or $22$ has derivations equal to zero by Lemma
\ref{lema222}, Lemma \ref{lema:2}, Lemma \ref{dij0} and Lemma \ref{lema:1}, in that order. On the other hand, the first column, the first row and the main diagonal of the derivation matrix associated to the evolution algebras of type $1$ is zero because of Lemma \ref{lema221}. Finally, the zeros in the derivation matrix associated to the evolution algebras of type $7$ and $19$ are consequence of Lemma \ref{lema:2}.


We point out that after the application of our results the task of identifying derivations is reduced to the study of only $5$ from the $23$ different evolution algebras whose directed graph is not twin-free. 

\newpage


\begin{center}
\begin{figure}[H]
\scalebox{0.80}{
\begin{tabular}{|c||c||c||c||C{3.4cm}||c|}
\hline
&&&&&\\
  Type & \rm{Number of arrows} & \rm{Twin class: $\T$} & $D^1(\T)$ & \rm{Graph} & \rm{Derivation} \\
 &&&&&\\
\hline
\hline
&&&&&\\[-0.2cm]
1
 &
3
 &
$\{i,j,k\}$
 &
$\{i\}$
 &

 \centerline{$\xymatrix@-1pc@R=10pt@C=0.4pt{
  		& \bullet_{i} \uloopr{} & & \\
  		\bullet_{j}\ar[ur] & &  \bullet_{k} \ar[ul]  \\
  		  	}$}

 &

$\begin{pmatrix}
0 & 0 & 0 \\
0 & 0 & d_{jk}  \\
0 & -\dfrac{\omega_{ki}}{\omega_{ji}}d_{jk}  & 0  \\
\end{pmatrix}$

 \\
&&&&&\\[-0.2cm]
\hline
\hline
&&&&&\\[-0.2cm]
2
 &
3
 &
$\{i,j\}$
 &
$\{i\}$
 &
  \centerline{$\xymatrix@-1pc@R=10pt@C=0.4pt{
  	& \bullet_{i} \uloopr{} & & \\
  		\bullet_{j}\ar[ur] & &  \bullet_{k} \ar[ll]  \\
  		  	}$}
  	
 &
$\begin{pmatrix}
0 & 0&0 \\
0 & 0 & 0 \\
0 & 0 & 0  \\
\end{pmatrix}$
\\
&&&&&\\[-0.2cm]
\hline
\hline
&&&&&\\[-0.2cm]

 3
 &
3
 &
$\{i,j\}$
 &
$\{k\}$
 &
  \centerline{$\xymatrix@-1pc@R=10pt@C=0.4pt{
  	& \bullet_{i} \ar[dr] & & \\
  		\bullet_{j}\ar@/^{-8pt}/[rr]& &  \bullet_{k} \ar@/^{-8pt}/[ll]  \\
  		  		}$}
 &

$\begin{pmatrix}
0 & 0&0 \\
0 & 0 & 0 \\
0 & 0 & 0  \\
\end{pmatrix}$

 \\
&&&&&\\[-0.2cm]
\hline
\hline
&&&&&\\[-0.2cm]

4
&
4
&
$\{i,j\}$
 &
$\{i\}$

&
\centerline{$\xymatrix@-1pc@R=10pt@C=0.4pt{
		& \bullet_{i} \uloopr{} & & \\
  		\bullet_{j}\ar[ur] & &  \bullet_{k} \ar[ll] \ar[ul] \\
  		  	}$}
&
$
\begin{pmatrix}
0 & 0&0 \\
0 & 0 & 0 \\
0 & 0 & 0  \\
\end{pmatrix}
$
\\

&&&&&\\[-0.2cm]
\hline
\hline
&&&&&\\[-0.2cm]

5
 &
4
 &
$\{i,j\}$
 &
$\{i\}$

&
\centerline{$\xymatrix@-1pc@R=10pt@C=0.4pt{
		& \bullet_{i} \uloopr{} & & \\
  		\bullet_{j}\ar[ur] & &  \bullet_{k} \uloopd{} \ar[ul] \\
  		  	}$}
&
$
\begin{pmatrix}
0 & 0&0 \\
0 & 0 & 0 \\
0 & 0 & 0  \\
\end{pmatrix}
$
 \\
&&&&&\\[-0.2cm]
\hline
\hline
&&&&&\\[-0.2cm]
6
 &
4
 &
$\{i,j\}$
 &
$\{i\}$

&
\centerline{$\xymatrix@-1pc@R=10pt@C=0.4pt{
		& \bullet_{i} \uloopr{} & & \\
  		\bullet_{j}\ar[ur] & &  \bullet_{k} \ar[ll] \uloopd{} \\
  		  	}$}
&
$
\begin{pmatrix}
0 & 0&0 \\
0 & 0 & 0 \\
0 & 0 & 0  \\
\end{pmatrix}
$

 \\
&&&&&\\[-0.2cm]
\hline
\hline
&&&&&\\[-0.2cm]
7
 &
4
 &
$\{i,j\}$
 &
$\{k\}$
 &
  \centerline{$\xymatrix@-1pc@R=10pt@C=0.4pt{
  	& \bullet_{i} \ar@/^{-5pt}/[dr] & & \\
  		\bullet_{j}\ar@/^{-5pt}/[rr]& &  \bullet_{k} \ar@/^{-5pt}/[ll] \ar@/^{-5pt}/[ul] \\
  		  		}$}
 &

$\begin{pmatrix}
d_{ii} & \dfrac{3\,\omega_{kj}}{\omega_{ki}}d_{ii} & 0 \\
\dfrac{3\,\omega_{ki}}{\omega_{kj}}d_{ii} & d_{ii} & 0 \\
0 & 0 & 2d_{ii} \\
\end{pmatrix}$
 \\
&&&&&\\[-0.2cm]
\hline
\hline
&&&&&\\[-0.2cm]
8
 &
4
 &
$\{i,j\}$
 &
$\{k\}$

&
\centerline{$\xymatrix@-1pc@R=10pt@C=0.4pt{
  	& \bullet_{i} \ar@/^{-5pt}/[dr] & & \\
  		\bullet_{j}\ar[rr]& &  \bullet_{k} \uloopd{}  \ar@/^{-5pt}/[ul] \\
  		  		}$}
 &
$
\begin{pmatrix}
0 & 0&0 \\
0 & 0 & 0 \\
0 & 0 & 0  \\
\end{pmatrix}
$
 \\
&&&&&\\[-0.2cm]
\hline
\hline
&&&&&\\[-0.2cm]
9
&
5
&
$\{i,j\}$
&
$\{i,j\}$
&
\centerline{$\xymatrix@-1pc@R=10pt@C=0.4pt{
  	& \bullet_{i} \ar@/^{-5pt}/[dl] \uloopr{}& & \\
  		\bullet_{j}\ar@/^{-5pt}/[ur] \dloopd{} & &  \bullet_{k}   \ar[ul] \\
  		  		}$}
&
$
\begin{pmatrix}
0 & 0&0 \\
0 & 0 & 0 \\
0 & 0 & 0  \\
\end{pmatrix}
$
 \\
&&&&&\\
\hline
\end{tabular}}
\end{figure}
\end{center}

\begin{center}
	\begin{figure}[H]
		\scalebox{0.70}{
			\begin{tabular}{|c||c||c||c||C{3.4cm}||c|}
				\hline
				&&&&&\\
				Type & \rm{Number of arrows} & \rm{Twin class: $\T$} & $D^1(\T)$ & \rm{Graph} & \rm{Derivation} \\
				&&&&&\\
				\hline
				\hline
				&&&&&\\[-0.2cm]
				
			10
			&
			5
			&
			$\{i,j\}$
 &
$\{i,k\}$

&
\centerline{$\xymatrix@-1pc@R=10pt@C=0.4pt{
  	& \bullet_{i} \ar@/^{-5pt}/[dr] \uloopr{} & & \\
  		\bullet_{j}\ar[rr] \ar[ur]& &  \bullet_{k}   \ar@/^{-5pt}/[ul] \\
  		  		}$}
 &
$
\begin{pmatrix}
0 & 0&0 \\
0 & 0 & 0 \\
0 & 0 & 0  \\
\end{pmatrix}
$
				
				\\
				&&&&&\\[-0.2cm]
				\hline
				\hline
				&&&&&\\[-0.2cm]
				
					11
					&
					5
					&
					$\{i,j\}$
 &
$\{i,k\}$

&
\centerline{$\xymatrix@-1pc@R=10pt@C=0.4pt{
  	& \bullet_{i} \ar[dr] \uloopr{} & & \\
  		\bullet_{j}\ar@/^{-5pt}/[rr] \ar[ur]& &  \bullet_{k}   \ar@/^{-5pt}/[ll] \\
  		  		}$}
 &
$
\begin{pmatrix}
0 & 0&0 \\
0 & 0 & 0 \\
0 & 0 & 0  \\
\end{pmatrix}
$

				\\
				&&&&&\\[-0.2cm]
				\hline
				\hline
				&&&&&\\[-0.2cm]

			12
				&
				5
				&
			$\{i,j\}$
 &
$\{i,k\}$

&
\centerline{$\xymatrix@-1pc@R=10pt@C=0.4pt{
  	& \bullet_{i} \ar[dr] \uloopr{} & & \\
  		\bullet_{j}\ar[rr] \ar[ur]& &  \bullet_{k} \uloopd{}   \\
  		  		}$}
 &
$
\begin{pmatrix}
0 & 0&0 \\
0 & 0 & 0 \\
0 & 0 & 0  \\
\end{pmatrix}
$

				\\
				&&&&&\\[-0.2cm]
				\hline
				\hline
				&&&&&\\[-0.2cm]
					13
					&
					6
					&
				$\{i,j,k\}$
				&
				$\{i,j\}$

					&
					\centerline{$\xymatrix@-1pc@R=10pt@C=0.4pt{
  	& \bullet_{i} \ar@/^{-5pt}/[dl] \uloopr{} & & \\
  		\bullet_{j}\ar@/^{-5pt}/[ur] \dloopd{}& &  \bullet_{k}  \ar[ll] \ar[ul]   \\
  		  		}$}
 &
$
\begin{pmatrix}
d_{ii} & \dfrac{d_{ii}\omega_{ij}}{\omega_{ii}}& d_{ik} \\
\dfrac{d_{ii}\omega_{ii}}{\omega_{ij}} & d_{ii} & -\dfrac{d_{ik}\omega_{ii}}{\omega_{ij}} \\
-\dfrac{d_{ik}\omega_{ki}}{\omega_{ii}} & \dfrac{d_{ik}\omega_{ki}}{\omega_{jj}} & d_{ii}  \\
\end{pmatrix}
$

				\\
				&&&&&\\[-0.2cm]
				\hline
				\hline
				&&&&&\\[-0.2cm]
				
					14
					&
					6
					&
				$\{i,j\}$
				&
				$\{i,j\}$

					&
					\centerline{$\xymatrix@-1pc@R=10pt@C=0.4pt{
  	& \bullet_{i} \ar@/^{-5pt}/[dl] \uloopr{} & & \\
  		\bullet_{j}\ar@/^{-5pt}/[ur] \dloopd{}& &  \bullet_{k}  \ar[ll] \uloopd{}   \\
  		  		}$}
 &
$
\begin{pmatrix}
0 & 0&0 \\
0 & 0 & 0 \\
0 & 0 & 0  \\
\end{pmatrix}
$
					
				\\
				&&&&&\\[-0.2cm]
				\hline
				\hline
				&&&&&\\[-0.2cm]
				
					15
					&
					6
					&
					$\{i,j\}$
				&
				$\{i,k\}$

					&
					\centerline{$\xymatrix@-1pc@R=10pt@C=0.4pt{
  	& \bullet_{i} \ar@/^{-5pt}/[dr] \uloopr{} & & \\
  		\bullet_{j}\ar@/^{-5pt}/[rr] \ar[ur]& &  \bullet_{k}  \ar@/^{-5pt}/[ll] \ar@/^{-5pt}/[ul]   \\
  		  		}$}
 &
$
\begin{pmatrix}
0 & 0&0 \\
0 & 0 & 0 \\
0 & 0 & 0  \\
\end{pmatrix}
$
			
				\\
				&&&&&\\[-0.2cm]
				\hline
				\hline
				&&&&&\\[-0.2cm]
				
				16
				&
				6
				&
				$\{i,j\}$
				&
				$\{i,k\}$

					&
					\centerline{$\xymatrix@-1pc@R=10pt@C=0.4pt{
  	& \bullet_{i} \ar[dr] \uloopr{} & & \\
  		\bullet_{j}\ar@/^{-5pt}/[rr] \ar[ur]& &  \bullet_{k}   \ar@/^{-5pt}/[ll] \uloopd{}  \\
  		  		}$}
 &
$
\begin{pmatrix}
0 & 0&0 \\
0 & 0 & 0 \\
0 & 0 & 0  \\
\end{pmatrix}
$

				\\
				&&&&&\\[-0.2cm]
				\hline
				\hline
				&&&&&\\[-0.2cm]
				
				17
				&
				7
				&
				$\{i,j\}$
				&
				$\{i,j,k\}$

					&
					\centerline{$\xymatrix@-1pc@R=10pt@C=0.4pt{
  	& \bullet_{i} \ar@/^{-5pt}/[dl] \ar[dr] \uloopr{} & & \\
  		\bullet_{j}\ar[rr] \ar@/^{-5pt}/[ur] \dloopd{}& &  \bullet_{k}   \uloopd{}  \\
  		  		}$}
 &
$
\begin{pmatrix}
0 & 0&0 \\
0 & 0 & 0 \\
0 & 0 & 0  \\
\end{pmatrix}
$
\\

				&&&&&\\
				\hline
				\hline
				&&&&&\\[-0.2cm]

			18
				&
				7
				&
			$\{i,j\}$
 &
$\{i,j,k\}$

&
\centerline{$\xymatrix@-1pc@R=10pt@C=0.4pt{
  	& \bullet_{i} \ar@/^{-5pt}/[dr] \ar@/^{-5pt}/[dl] \uloopr{} & & \\
  		\bullet_{j}\ar[rr] \dloopd{} \ar@/^{-5pt}/[ur]& &  \bullet_{k}   \ar@/^{-5pt}/[ul]\\
  		  		}$}
 &
$
\begin{pmatrix}
0 & 0&0 \\
0 & 0 & 0 \\
0 & 0 & 0  \\
\end{pmatrix}
$

				\\
				&&&&&\\
				\hline
			\end{tabular}}
		\end{figure}
	\end{center}
	
	\begin{center}
	\begin{figure}[H]
		\scalebox{0.70}{
			\begin{tabular}{|c||c||c||c||C{3.4cm}||c|}
				\hline
				&&&&&\\
				Type & \rm{Number of arrows} & \rm{Twin class: $\T$} & $D^1(\T)$ & \rm{Graph} & \rm{Derivation} \\
				&&&&&\\[-0.2cm]
				\hline
				\hline
				&&&&&\\[-0.2cm]
					
			19
			&
			7
			&
			$\{i,j\}$
 &
$\{i,j\}$

&
\centerline{$\xymatrix@-1pc@R=10pt@C=0.4pt{
  	& \bullet_{i} \ar@/^{-5pt}/[dl] \uloopr{} & & \\
  		\bullet_{j} \ar@/^{-5pt}/[ur] \dloopd{}& &  \bullet_{k}   \ar[ul] \ar[ll] \uloopd{}\\
  		  		}$}
 &
$
\begin{pmatrix}
-\dfrac{\omega_{jj}d_{ij}}{\omega_{ii}} & d_{ij}&0 \\
-\dfrac{\omega_{ji}d_{ij}}{\omega_{ii}} & \dfrac{\omega_{ji}d_{ij}}{\omega_{jj}} & 0 \\
0 & 0 & 0  \\
\end{pmatrix}
$
				
				\\
				&&&&&\\[-0.2cm]
				\hline
				\hline
				&&&&&\\[-0.2cm]
				
					20
					&
					7
					&
					$\{i,j\}$
 &
$\{i,k\}$

&
\centerline{$\xymatrix@-1pc@R=10pt@C=0.4pt{
  	& \bullet_{i} \ar@/^{-5pt}/[dr] \uloopr{} & & \\
  		\bullet_{j}\ar@/^{-5pt}/[rr] \ar[ur]& &  \bullet_{k} \uloopd{}  \ar@/^{-5pt}/[ll] \ar@/^{-5pt}/[ul]\\
  		  		}$}
 &
$
\begin{pmatrix}
0 & 0&0 \\
0 & 0 & 0 \\
0 & 0 & 0  \\
\end{pmatrix}
$

				\\
				&&&&&\\[-0.2cm]
				\hline
				\hline
				&&&&&\\[-0.2cm]

			21
				&
				8
				&
			$\{i,j\}$
 &
$\{i,j,k\}$

&
\centerline{$\xymatrix@-1pc@R=10pt@C=0.4pt{
  	& \bullet_{i} \ar@/^{-5pt}/[dr] \ar@/^{-5pt}/[dl] \uloopr{} & & \\
  		\bullet_{j}\ar@/^{-5pt}/[rr] \dloopd{} \ar@/^{-5pt}/[ur]& &  \bullet_{k}   \ar@/^{-5pt}/[ll]  \ar@/^{-5pt}/[ul]\\
  		  		}$}
 &
$
\begin{pmatrix}
0 & 0&0 \\
0 & 0 & 0 \\
0 & 0 & 0  \\
\end{pmatrix}
$

				\\
				&&&&&\\[-0.2cm]
				\hline
				\hline
				&&&&&\\[-0.2cm]
					22
					&
					8
				&
			$\{i,j\}$
 &
$\{i,j,k\}$

&
\centerline{$\xymatrix@-1pc@R=10pt@C=0.4pt{
  	& \bullet_{i} \ar@/^{-5pt}/[dr] \ar@/^{-5pt}/[dl] \uloopr{} & & \\
  		\bullet_{j}\ar[rr] \dloopd{} \ar@/^{-5pt}/[ur]& &  \bullet_{k}  \ar@/^{-5pt}/[ul]\uloopd{}\\
  		  		}$}
 &
$
\begin{pmatrix}
0 &0 &0 \\
0 & 0 & 0 \\
0 & 0 & 0  \\
\end{pmatrix}
$

				\\
				&&&&&\\[-0.2cm]
				\hline
				\hline
				&&&&&\\[-0.2cm]
				
					23
					&
					9
					&
				$\{i,j,k\}$
				&
				$\{i,j,k\}$

					&
					\centerline{$\xymatrix@-1pc@R=10pt@C=0.4pt{
  	& \bullet_{i} \ar@/^{-5pt}/[dr] \ar@/^{-5pt}/[dl] \uloopr{} & & \\
  		\bullet_{j}\ar@/^{-5pt}/[rr] \dloopd{} \ar@/^{-5pt}/[ur]& &  \bullet_{k}   \ar@/^{-5pt}/[ll]  \ar@/^{-5pt}/[ul] \uloopd{}\\
  		  		}$}
 &
$
$
See $\ast$
$ 
$

				\\
				&&&&&\\
				\hline
			\end{tabular}}
		\end{figure}
	\end{center}

$\ast$

\footnotesize{
$\begin{pmatrix}
-\left(\dfrac{\omega_{jj}d_{ij}}{\omega_{ii}}+\dfrac{\omega_{kk}d_{ik}}{\omega_{ii}}\right) & d_{ij}& d_{ik} \\
-\dfrac{\omega_{jj}d_{ij}}{\omega_{ij}} & -\left(\dfrac{\omega_{jj}d_{ij}}{\omega_{ii}}+\dfrac{\omega_{kk}d_{ik}}{\omega_{ii}}\right) & -\left(\dfrac{\omega_{jk}d_{ij}}{\omega_{ii}}+\dfrac{(\omega_{jk}\omega_{kk}+\omega_{ji}\omega_{ii})d_{ik}}{\omega_{ii}\omega_{jj}}\right) \\
-\dfrac{\omega_{kk}d_{ik}}{\omega_{ik}} & \dfrac{-\omega_{jj}}{\omega_{jk}}\left(\dfrac{\omega_{jj}}{\omega_{ii}}+\dfrac{\omega_{ji}}{\omega_{jj}}\right)d_{ij}- \dfrac{\omega_{kk}d_{ik}}{\omega_{ii}} & -\left(\dfrac{\omega_{jj}d_{ij}}{\omega_{ii}}+\dfrac{\omega_{kk}d_{ik}}{\omega_{ii}}\right)  \\
\end{pmatrix}$	}
	
	\par


\noindent

\bigskip
\begin{thebibliography}{99}




\bibitem{Alsarayreh/Qaralleh/Ahmad/2017}
Alsarayreh A., Qaralleh I., Ahmad M. Z. Derivation of three dimensional evolution algebra. JP J. Algebra Number Theory Appl. 2017; No. 39(4):425-444.


\bibitem{YMV}
Cabrera Y., Siles M., Velasco M.V. Evolution algebras of arbitrary dimension and their decompositions. Linear Algebra Appl. 2016;  No. 496:122-162.

\bibitem{YMV2}
Cabrera C.Y., Siles M. M., Velasco M.V. Classification of three-dimensional evolution algebras. Linear Algebra Appl. 2017; No. 524:68-108.


\bibitem{PMP}
Cadavid P.,  Rodi\~no Montoya M. L., Rodriguez P. M. The connection between evolution algebras, random walks, and graphs. J. Algebra Appl., DOI: 10.1142/S0219498820500231.

\bibitem{PMP2}
Cadavid P.,  Rodi\~no Montoya M. L., Rodriguez P. M. On the isomorphisms between evolution algebras of graphs and random walks, Linear Multilinear Algebra, DOI: 10.1080/03081087.2019.1645807.

\bibitem{PMP3}
Cadavid P.,  Rodi\~no Montoya M. L., Rodriguez P. M. Characterization theorems for the spaces of derivations of evolution algebras associated to graphs. Linear Multilinear Algebra, DOI: 10.1080/03081087.2018.1541962.


\bibitem{camacho/gomez/omirov/turdibaev/2013}
Camacho L. M., G\'omez J. R., Omirov B. A., Turdibaev R. M. Some properties of evolution algebras. Bull. Korean Math. Soc. 2013; 50 No. 5:1481-1494. 

\bibitem{COT}
Camacho L. M., G\'omez J. R., Omirov B. A., Turdibaev R. M. The derivations of some evolution algebras. Linear Multilinear Algebra 2013; No. 61:309-322.

\bibitem{cardoso}
Cardoso M.I., Gonçalves D., Mart\'in D., Mart\'in C., Siles M., Squares and associative representations of two dimensional evolution algebras, arXiv:1807.02362, 2018.



\bibitem{casas/ladra/omirov/rozitov/2013}
Casas J.M., Ladra M., Omirov B.A., Rozikov U.A.. On nilpotent index and dibaricity of evolution algebras. Linear Algebra Appl. 2013; No. 439(1):90–105.

\bibitem{casas/ladra/rozitov/2011}
Casas J.M., Ladra M., Rozikov U.A. A chain of evolution algebras. Linear Algebra Appl. 2011; No. 435(4):852–870.


\bibitem{costa1}
 Costa R. On the derivations of gametic algebras for polyploidy with multiple alleles. Bol. Soc. Brasil. Mat. 1982; 13 No. 2:69-81.
 
\bibitem{costa2}
Costa R. On the derivation algebra of zygotic algebras for polyploidy with multiple alleles. Bol. Soc. Brasil. Mat. 1983; 14 No. 1:63-80.


\bibitem{Elduque/Labra/2015}
Elduque A., Labra A. Evolution algebras and graphs. J. Algebra Appl. 2015; No. 14:1550103.

\bibitem{Elduque/Labra/2019}
Elduque A., Labra A. Evolution algebras, automorphisms, and graphs. Linear Multilinear Algebra, DOI: 10.1080/03081087.2019.1598931.

\bibitem{falcon}
Falc\'on O.J. , Falc\'on R.M. , Nu\~nez J. Algebraic computation of genetic patterns related to three-dimensional evolution algebras. Appl Math Comput. 2018; No. 319:510–517.

\bibitem{gonshor}
Gonshor H. Derivations in genetic algebras. Comm. Algebra 1988; 16 No. 8:1525-1542.

\bibitem{gonzales}
Gonzalez S., Martinez C. Bernstein algebras with zero derivation algebra. Linear Algebra Appl. 1993; No. 191:235-244. 

\bibitem{holgate}
Holgate P. The interpretation of derivations in genetic algebras. Linear Algebra Appl. 1987; No. 85:75-79.





\bibitem{Mukhamedov/Qaralleh/2014}
Mukhamedov F., Qaralleh I. On Derivations Of Genetic Algebras. J. Phys.: Conf. Ser. 2014, No. 553:012004.

\bibitem{Mukhamedov/Khakimov/Omirov/Qaralleh/2019}
Mukhamedov F., Khakimov O., Omirov B., Qaralleh I. Derivations and automorphisms of nilpotent evolution algebras with maximal nilindex. J. Algebra Appl. 2019, No. 18(12):1950233.


\bibitem{paniello}
Paniello I. Evolution coalgebras. Linear Multilinear Algebra 2019, No. 67(8):1539-1553. 

\bibitem{peresi}
Peresi, L. A. The derivation algebra of gametic algebra for linked loci. Math. Biosci. 1988; 91 No. 2:151-156.

\bibitem{tian}
Tian J. P. Evolution algebras and their applications. Springer-Verlag Berlin Heidelberg, 2008.

\bibitem{tian2}
Tian J. P. Invitation to research of new mathematics from biology: evolution algebras. Topics in functional analysis and algebra, Contemp. Math. No. 672, 257-272, Amer. Math. Soc., Providence, RI, 2016.

\bibitem{tv}
Tian J.P., Vojtechovsky P. Mathematical concepts of evolution algebras in non-Mendelian genetics. Quasigroups Related Systems 2006; (1) No. 14:111-122.

\end{thebibliography}
\end{document}